\documentclass[12pt,letterpaper]{article}
\usepackage{graphicx,color}
\usepackage[cp850]{inputenc}
\usepackage[T1]{fontenc}
\usepackage[english]{babel}
\usepackage{cite}
\usepackage{rotating}
\usepackage[f]{esvect}
\usepackage{amsmath, amsthm, amssymb}
\usepackage{rotating}
\usepackage{lscape}
\usepackage{mathrsfs}
\usepackage{subcaption}
\usepackage{mathtools}

\def\tablenotes{\bgroup\parfillskip=0pt plus 1fil
\leftskip=0pt\relax \rightskip=0pt
\vskip2pt\footnotesize}
\def\endtablenotes{\vskip1pt\egroup}

\def\sphline{\noalign{\vskip3pt}\hline\noalign{\vskip3pt}}

\newtheorem{theorem}{Theorem}[section]

\newtheorem{corollary}[theorem]{Corollary}

\newtheorem{definition}[theorem]{Definition}

\allowdisplaybreaks

\begin{document}
\title{On the Computation of Eigenvalues of the Anharmonic Coulombic Potential}
\author{Tyler Cassidy, Philippe Gaudreau and Hassan Safouhi\footnote{Corresponding author. \newline The corresponding author acknowledges the financial support for this research by the Natural Sciences and Engineering Research Council of Canada~(NSERC) - Grant 250223-2011.}
\\
{\it Mathematical Section}\\
{\it Campus Saint-Jean, University of Alberta}\\
{\it 8406, 91 Street, Edmonton (AB), Canada T6C 4G9}
}

\date{}
\maketitle


\vspace*{0.75cm}
{\bf \large Abstract}.

In this work, we propose a method combining the Sinc collocation method with the double exponential transformation for computing the eigenvalues of the anharmonic Coulombic potential. We introduce a scaling factor that improves the convergence speed and the stability of the method. Further, we apply this method to Coulombic potentials leading to a highly efficient and accurate computation of the eigenvalues.

\vspace*{0.5cm}
{\bf Keywords}

Coulombic anharmonic potentials. Schr\"{o}dinger equation. Sinc collocation method. Double exponential transformation.

\clearpage
\section{Introduction}
The Coulombic anharmonic oscillator potential has been of considerable interest in the study of the Schr\"{o}dinger equation. The potential describes the interaction between charged particles and consistently arises in physical applications. These applications include  interactions in atomic, molecular and particle physics, and between nuclei in plasma \cite{Roychoudhury1990a,Ozer2003}. The study of the Schr\"{o}dinger equation involves computation of the energy states, and many different methods have been proposed for accurate and efficient calculation of the energy eigenvalues \cite{Ozer2003,Chaudhuri1995a,Ikhdair2007,Fernandez2008a, Lund1984}. In \cite{Chaudhuri1995a}, the authors use the Hill determinant method to numerically evaluate the Coulomb potential in N dimensions. They initially transform the N dimension differential equation into a $(2N-4)$ dimensional problem. This transformation produces the structure of a one dimensional Schr\"{o}dinger equation with a spherically symmetric potential. The Hill determinant then produces numerical approximations of the energy eigenvalues. In \cite{Ikhdair2007}, the authors utilize the wavefunction and the Hill determinant method to find energy eigenvalues for the Coulomb potential and the sextic oscillator problem. They also produce a relation between parameters leading to exactly solvable equations. However, the Hill determinant method presents several limitations, including a lack of convergence to higher order eigenvalues and the production of non physically realistic results \cite{Tater1999}. In addition, the method does not account for an important aspect of the wavefunction, namely decay at the boundaries~\cite{Tater1999}.

Conversely, the Riccati-Pad\'{e} method has been used in the calculation of bound states and resonances in the Coulomb potential \cite{Fernandez2008a}. This method consists of transforming the Schr\"{o}dinger equation into a Riccati type equation for the logarithmic derivative of the wavefunction. Analysis of the Riccati equation provides a deeper understanding of the overall nature of both the wavefunction and the energy eigenvalues. In \cite{Fernandez2008a}, the method shows convergence towards eigenvalues of the Schr\"{o}dinger equation for bound and unbound states. In a separate work \cite{Fernandez1991b}, the Riccati-Pad\'{e} method is combined with Hankel determinants to find resonance states of the Coulomb potential. While useful, the Riccati-Pad\'{e} method can only produce bounds on the eigenvalues. These bounds can give quite good approximations of the energy eigenvalues, but can also be so large that they do not produce any meaningful information \cite{Fernandez1989}. Achieving acceptable error bounds on the eigenvalues requires an increase in the dimension of the Hankel determinants. Further, the complexity of the method increases with the complexity of the potential. Finally, the method can also yield unwanted and unrealistic solutions \cite{Fernandez1989}.

The super symmetric quantum mechanic approach has also produced results with potentials of the form $ V(x) = \frac{\alpha}{r} + \sum_{i=1}{4}p_ir^i$. In~\cite{Roychoudhury1990a}, the authors solved the equation using supersymmetric quantum mechanics. Their results are mostly in agreement with exact values. Nevertheless, poor agreement seems to arise when the potential has multiple wells or roots. There has also been advancement in the combination of supersymmetric quantum mechanics and perturbation theory. In~\cite{Ozer2003}, a combination of these techniques to find exact solutions to the perturbed Coulomb potential is proposed. This method can be expanded to include many other potentials and their excited states. However, the method requires constraints on the parameters of the potential and these constraints differ for different eigenvalues \cite{Ozer2003}.

In~\cite{Lund1984}, the Sinc collocation method (SCM) have been used in a combination with the single exponential (SE) transformation  to compute the energy eigenvalues of the radial Schr\"{o}dinger equation. The Sinc function and Sinc collocation method have been used extensively since their introduction to solve a variety of numerical problems~\cite{Stenger-33-85-79, Stenger-23-165-81, Stenger-121-379-00}. The applications include numerical integration, linear and non-linear ordinary differential equations as well as partial differential equations~\cite{Gaudreau2014, Gaudreau2014a, Al-Khaled-283-2001, Al-Khaled-245-2001, Carlson-Dockery-Lund-66-215-97, Amore-39-L349-06, ElGamel-Zayed-48-1285-04, ElGamel-Cannon-Zayed-73-1325-03, Smith-Bogar-Bowers-Lund-28-760-91, Eggert1987}. The single exponential Sinc collocation method (SESCM) has been shown to offer an exponential convergence rate and works well in the presence of singularities. The double exponential (DE) transform, introduced in 1974~\cite{Takahasi-Mori-9-721-74}, yields near optimal accuracy when using the trapezoidal rule in numerical integration \cite{Mori2001a, Sugihara2002a}. Since the introduction of the DE transform, its effectiveness has been studied extensively~\cite{Tanaka2009, Sugihara2002b}. While exponential convergence is produced using the SESCM, it has been shown that the double exponential transformation provides an improved numerical convergence~\cite{Okayama2013, Tanaka2013, Mori2005}.

The combination of SCM with the DE transformation was used to compute eigenvalues of the anharmonic oscillator $V(x) = \sum_{i=1}^n a_i x^{2i}$~\cite{Gaudreau2014} and to Sturm-Liouville boundary value problems~\cite{Gaudreau2014a}. This method which is referred to as DESCM is shown to be highly accurate, efficient and stable for computing the energy eigenvalues of the Schr\"{o}dinger equation. In~\cite{Gaudreau2014}, an optimal mesh size for potentials with multiple wells was derived leading to a substantial improvement of the convergence of the method.

In this work, we provide a refinement for the DESCM and we apply the metod to the anharmonic Coulombic potential. The improved method is capable of dealing  efficiently with a vast variety of potentials. The DESCM approximates the wavefunction with a series of weighted Sinc functions. By substituting the approximation into the Schr\"{o}dinger equation, we obtain a generalized eigensystem where the generalized eigenvalues are approximations to the exact energy eigenvalues. We preform asymptotic analysis on the Schr\"{o}dinger equation with the anharmonic Coulombic potential. We use the asymptotic solutions to produce optimized double exponential transformations. We also present a numerical scaling that improves both the numerical convergence and stability of the method. Finally, we compare the results of the refined DESCM with the SESCM to illustrate the superiority of the proposed refinement.

\section{Definitions and properties}
The sinc function is defined by the following expression:
\begin{equation} \label{formula: sinc functions}
\textrm{sinc}(z) = \dfrac{\sin(\pi z)}{\pi z},
\end{equation}
where $z \in \mathbb{C}$ and the value at $z=0$ is taken to be the limiting value ${\rm sinc}(0)=1$.

For $j \in \mathbb{Z}$ and $h$ a positive number, we define the Sinc function $S(j,h)(x)$ by:
\begin{equation}
S(j,h)(x) = \textrm{sinc}\left( \dfrac{x}{h}-j\right).
\end{equation}

We also note the discrete orthogonality of the Sinc functions \cite{Stenger-121-379-00}:
\begin{equation}
S(j,h)(k\, h) =
\left \{
\begin{array} {ccc}
1 & \textrm{if} & k = j \\
0 & \textrm{if} & k \not= j.
\end{array}
\right.
\end{equation}

\begin{definition}
Given a function g: $\mathbb{R} \to \mathbb{R}$ and any $h$ positive, the Sinc expansion also known as Whittaker Cardinal expansion of g is defined as:
\begin{equation}
C(g,h)(x) = \displaystyle \sum_{j=-\infty}^{\infty} g(jh)S(j,h)(x).
\end{equation}
\end{definition}

\begin{definition}
Let $d >0$ and consider the set $\mathcal{D}_{d} $ be a strip of width $d$ about the real axis defined as following:
\begin{equation}
\mathcal{D}_{d} = \left \{ z \in \mathbb{C} \;\;\textrm{such that} \;\; \left\| \Im(z) \right\|  < d \right \}.
\end{equation}
We also define a rectangle in $\mathbb{C}$ such that, for $ \varepsilon \in (0,1)$:
\begin{equation}
\mathcal{D}_{d}(\varepsilon ) = \left \{ z \in \mathbb{C} \;\;\textrm{such that} \;\; \left\| \Im(z) \right\| <d(1-\varepsilon) \right \}.
\end{equation}
\end{definition}

We shall now present a space of functions well suited to Sinc approximation.
\begin{definition}
Let $\mathbf{B}_p(\mathcal{D}_{d})$ be the family of functions $g$ that are analytic in $\mathcal{D}_{d}$ and such that:
\begin{equation}
\int_{-d}^{d} |g(x+iy)|  \textrm{d}y \to 0 \quad \textrm{as} \quad |x| \to \infty .
\end{equation}

Additionally, for $ p \geq 1$,  $\mathcal{N}_p(g,\mathcal{D}_{d}) < \infty$ where:
\begin{equation}
\mathcal{N}_p(g,\mathcal{D}_{d}) =
\left \{
\begin{array} {lll}
\displaystyle \lim_{\varepsilon \to 0} \left( \int_{\partial \mathcal{D}_{d}(\varepsilon )} |g(z)|^p |dz| \right) ^{\frac{1}{p}} & \textrm{if} & 1 \leq p < \infty \\[0.5cm]
\displaystyle \lim_{\varepsilon \to 0} \sup_{z \in \mathcal{D}_{d}(\varepsilon)} |g(z)| & if & p = \infty .
\end{array}
\right.
\end{equation}
\label{Deffunctionspace}
\end{definition}

An analysis of the error induced when approximating a function with a Sinc expansion can be found in \cite{Stenger-23-165-81}.

Eggert et al.~\cite{Eggert1987} proposed a change of variable to produce a symmetric discretized system when employing the Sinc collocation method. The change of variable is given by:
\begin{equation}
v(x) = \left( \sqrt{ (\phi^{-1})'} \, \psi\right) \circ  \phi(x),
\label{formula: EggertSub}
\end{equation}
where the conformal map $\phi(x)$ is defined according to the following definition.

\begin{definition}
\cite{Eggert1987}
Let $ \Omega_{d} $ be a simply connected domain in the complex plane with boundary points $a$ and $b$. Define a conformal map, $\phi^{-1}$, from $\Omega_{d} $ onto the infinite strip $ \mathcal{D}_{d}$ with $ \phi^{-1}(a) = -\infty$ and $ \phi^{-1}(b) = \infty$. Denote the inverse of $ \phi^{-1}$ by $\phi$.
\label{Deftransform}
\end{definition}

To utilize the optimality of the double exponential transformation~\cite{Sugihara2002a}, we search for a change of variable that will result in $v(x)$ given by~\eqref{formula: EggertSub} to decay double exponentially.

\section{The DESCM and the Coulombic anharmonic potential}
We consider the Schr\"{o}dinger equation with the boundary conditions:
\begin{eqnarray}
\mathcal{H} \, \psi(x) &  = &  E \, \psi(x) \quad \textrm{for} \quad  0<x< \infty
\nonumber\\
\psi(0) & = & \psi(\infty) = 0,
\end{eqnarray}
where the Hamiltonian operator $\mathcal{H} = - \dfrac{d^2}{dx^2} + V(x)$ and the Coulombic anharmonic potential $V(x)$ is given by:
\begin{eqnarray}
V(x) & = & \frac{a_{-2}}{x^2}+\frac{a_{-1}}{x} + \displaystyle \sum_{i=1}^n a_i x^i
\nonumber\\ & = & \displaystyle \sum_{i=-2}^{n} a_i x^{i} \qquad \textrm{with} \quad a_{-2} >0, \; a_0 = 0 \;\; \textrm{and} \;\; a_n >0.
\label{EQPotential}
\end{eqnarray}

The negative powers of $x$ and the singularity at $x = 0$ are some of the defining features of the anharmonic Coulombic potential.

We use the transformation~\eqref{formula: EggertSub}, such that the function $v(x)$ decays double exponentially. This produces the following transformed differential equation:
\begin{equation}
- v^{\prime \prime}(x) + \tilde{V}(x) v(x) = E\, ( \phi '(x))^2 v(x) \quad \textrm{with} \quad \lim_{|x|\to \infty} v(x) = 0,
\label{EQtransformedSchrodinger}
\end{equation}
where:
\begin{equation}
\tilde{V}(x) = -\sqrt{ \phi'(x)} \frac{d}{dx} \left( \frac{1}{\phi'(x)}\frac{d}{dx} \sqrt{ \phi'(x)} \right) + (\phi'(x))^2 \, V(\phi(x)).
\end{equation}

To approximate the solution using the Sinc collocation method, we use the truncated Sinc expansion given by:
\begin{equation}
C_{N}(v,h)(x) = \displaystyle \sum_{j=-N}^{N} v_{j}S(j,h)(x) \qquad \textrm{with} \quad v_{j} = v(jh),
\end{equation}
where $h$ is the mesh size.

Inserting the Sinc expansion of the wavefunction in the differential equation leads to:
\begin{equation}
 \displaystyle \sum_{j = -N}^{N} \left[\frac{-1}{h^2} \delta_{j,k}^{(2)} + \tilde{V}(kh)\delta_{j,k}^{(0)}\right] v_{j} =
{\mathcal E} \, \displaystyle \sum_{j=-N}^{N} \delta_{j,k}^{(0)} ( \phi(x))^2 v_{j},
\label{EQSincODE}
\end{equation}
where:
\begin{equation}
\delta_{j,k}^{(2)}  =
\left \{
\begin{array} {lll}
\dfrac{ - \pi^2}{3} & \textrm{if} & j=k \\[0.25cm]
\dfrac{(-2)(-1)^{j-k}}{(j-k)^2} & \textrm{if} & j \neq k
\end{array}
\right.
\qquad
\textrm{and}
\qquad
\delta_{j,k}^{(0)}  =
\left \{
\begin{array} {lll}
0 & \textrm{if} & k \neq j \\[0.25cm]
1 & \textrm{if} & k = j. \\
\end{array}
\right.
\end{equation}

In~(\ref{EQSincODE}), the value ${\mathcal E}$ is an approximation of the exact energy eigenvalue $E$ of the system~(\ref{EQtransformedSchrodinger}).

Using the identities for $\delta_{j,k}^{(0)}$ and $\delta_{j,k}^{(2)}$, we are able to rewrite (\ref{EQSincODE}) in a matrix form as:
\begin{equation}
\mathcal{H} \, \mathbf{C}_N(v,h) = \mathbf{H\,}\mathbf{v} = {\mathcal E}\, \mathbf{D}\mathbf{v}  \quad \Rightarrow \quad (\mathbf{H}-{\mathcal E}\, \mathbf{D})\, \mathbf{v} = 0,
\label{EQMatrixODE}
\end{equation}
where we define:
\begin{equation}
 {\bf v}= [ v(-Nh),...,v(Nh) ]^{t} \quad \textrm{and} \quad  \mathbf{C}_N(v,h) = [ C_N(v,h)(-Nh),..., C_N(v,h)(Nh)]^{t}.
\end{equation}

The matrix $\mathbf{H}$ and the diagonal matrix $\mathbf{D}$ are given by:
\begin{equation}
\mathbf{H}_{j,k} = \frac{-1}{h^2} \delta_{j,k}^{(2)} + \tilde{V}(kh)(kh) \delta_{j,k}^{(0)} \qquad \textrm{and} \qquad \mathbf{D}_{j,k} = ( \phi'(x))^2 \delta_{j,k}^{(0)}.
\label{matrixdefine}
\end{equation}

As can be seen from (\ref{EQMatrixODE}), the differential equation can be transformed into a generalized eigenvalue problem. The non-trivial solutions are approximations for the exact energy eigenvalues of the system. These non-trivial solutions are the roots of $\det(\mathbf{H}- {\mathcal E} \, \mathbf{D})$.

To utilize the Sinc collocation method to compute eigenvalues of Coulombic potential, we search for an appropriate double exponential transform as defined in Definition \ref{Deftransform}. To find such a transformation, we perform asymptotic analysis of the differential equation.

As $x \to \infty$, the potential is dominated by the $x^n$ term and the differential equation becomes:
\begin{equation}
-\psi''(x) + a_n \, x^n \, \psi(x) \sim 0 \quad  \textrm{as}  \quad x \to \infty.
\end{equation}

Making the substitution $\psi(x) = e^{S(x)}$ where $S(x)$ is such that $S''(x) = o(S'(x)^2)$ since $n \geq 2$, in the above equation leads to:
\begin{equation}
-S'(x)^2 +a_n x^n \sim 0 \quad  \textrm{as}  \quad x \to \infty.
\end{equation}

Solving the above equation and taking the negative root to satisfy the boundary conditions, we obtain:
\begin{equation}
S(x) \sim - \frac{2 \sqrt{a_n}}{n+2} x^{\frac{n+2}{2}} \quad  \textrm{as}  \quad x \to \infty.
\end{equation}

To find the second order approximation, we let:
\begin{equation}
S(x) = - \frac{2 \sqrt{a_n}}{n+2} x^{\frac{n+2}{2}} + C(x)  \qquad \textrm{with} \qquad C(x) = o(x^{\frac{n+2}{2}}).
\end{equation}

Substituting $\psi(x) = e^{S(x)}$ into the differential equation and canceling higher order terms gives the WKB expansion of the wavefunction at infinity:
\begin{equation}
\psi(x) \sim  A\,x^{-\frac{n}{4}} \exp \left( - \frac{2 \sqrt{a_n}}{n+2} x^{\frac{n+2}{2}} \right) \quad  \textrm{as}  \quad x \to \infty.
\label{EQasmyptoticsolution01}
\end{equation}

Conversely, as $ x \to 0^{+}$, the Coulomb potential is dominated by $\dfrac{a_{-2}}{x^2}$. We see that $ x = 0$ is a regular singular point and the equation requires the use of Frobenius series type solutions.  The solution is of the form $\psi(x) \sim x^r$, where $r$ is a solution of the indicial equation:
\begin{equation}
-r(r-1)+ a_{-2} = 0 \quad \Longrightarrow \quad r = \frac{1\pm  \sqrt{1+4a_{-2}}}{2}.
\end{equation}

Using $\psi(0)=0$, we obtain:
\begin{equation}
\psi(x) \sim A_{2} x^{\frac{1+\sqrt{1+4a_{-2}}}{2}} \quad \textrm{as} \quad x \to 0^{+}.
\label{EQPSI01}
\end{equation}

Finally, we notice that the wavefunction exhibits exponential decay at infinity and polynomial decay at zero. The following development mirrors the work done in \cite{Eggert1987}. We search for a transformation $\phi(x)$ that satisfies Definition \ref{Deftransform} and produces double exponential decay at the boundaries. We begin by using the transformation proposed in \cite{Tanaka2009}:
\begin{equation}
\phi(x) = \log\left[e^{\sinh(x)}+1\right] \,\sim \,
\left\{
\begin{array}{lll}
\dfrac{e^x}{2} & \textrm{as} & x \to \infty\\[0.25cm]
\exp\left[\dfrac{-e^{-x}}{2}\right] & \textrm{as} & x \to -\infty.
\end{array} \right.
\label{EQphix}
\end{equation}

\begin{theorem} \cite{Gaudreau2014}
Consider the transformed Schr\"{o}dinger equation.
\begin{equation}
- v''(x) + \tilde{V}(x) v(x) = E ( \phi'(x))^2 v(x) \quad
\textrm{with} \quad \lim_{|x|\to \infty} v(x) = 0,
\end{equation}
where:
\begin{equation}
\tilde{V}(x) = -\sqrt{ \phi'(x)} \frac{d}{dx} \left( \frac{1}{\phi'(x)}\frac{d}{dx} \sqrt{ \phi'(x)} \right) + (\phi'(x))^2 \, V(\phi(x)).
\end{equation}

Further, suppose that the transformation is such that:
\begin{equation}
|v(x)| \leq A \exp\left(-B e^{\gamma |x|}\right),
\end{equation}
for some positive constants $A,B,\gamma$. Take $ d = \dfrac{\pi}{2 \gamma}$ and suppose further that there exists a $\delta>0$ such that $\tilde{V}(x) \geq \delta^{-1}$. Then, the optimal distance between collocation points is given by:
\begin{equation}
h = \frac{W(\pi d \gamma N/B)}{\gamma N},
\end{equation}
where $W(z)$ is the Lambert W Function. Moreover, the convergence of numerical eigenvalues ${\mathcal E}$ to exact eigenvalues $E$ is given by:
\begin{equation}
|E - {\mathcal E}| \leq \Theta_{d,v} \sqrt{d E} \left( \frac{N^{\frac{5}{2}}}{\log(N)} \right)\exp \left(- \frac{\pi d \gamma N}{\log( \pi d \gamma N /B)} \right),
\end{equation}
$\Theta_{d,v}$ is a constant that depends on $d$ and $v(x)$.
\end{theorem}

\section{Refinement of the DESCM}
We notice that we have a singularity at the left end point of the potential at $ x = 0$. Following the approach detailed in \cite{Mori1988}, we search for a general transformation of the form:
\begin{equation}
\tilde{\phi}(x) = \log\left[\exp\left(ae^{b x}-ce^{-d x}\right)+1\right].
\label{EQtildehix}
\end{equation}

We find that the transformation $\tilde{\phi}(x)$ is ideal as it produces double exponential decay at both boundaries and is suitable to Sinc expansion. As $x \to \infty$:
\begin{equation}
\log\left[\exp\left(a e^{b x}-ce^{-d x} \right)+1\right] \sim  a e^{b x}.
\label{EQasymptoticexpansion01}
\end{equation}

We insert the asymptotic expansion of the transformation in the asymptotic approximation of the wavefunction and find that:
\begin{equation}
v(x) \sim A\left(e^{b x}\right)^{\frac{n}{4}} \exp \left[ - \frac{2 \sqrt{a_n}}{n+2} \left(a e^{b x}\right)^{\frac{n+2}{2}} \right] \quad \textrm{as} \quad x \to \infty.
\end{equation}

Let $x \,=\, \alpha + i \beta \,\in\, \mathbb{C}$ and as $ \alpha \to \infty$:
\begin{eqnarray}
|v(x)| & \sim  & \left|A \exp\left[ \frac{nb}{4} x - Be^{\frac{b x (n+2)}{2}}\right]\right|
\nonumber\\ & \sim & \left|A \exp\left[ \frac{nb}{4}(\alpha + i \beta) - Be^{\frac{b}{2} (\alpha + i \beta)}\right] \right|
\nonumber\\ & \sim & \left|A \exp\left[ \frac{nb}{4}(\alpha + i \beta)- Be^{\frac{b}{2} \alpha}\left( \cos\left( \frac{b}{2} \beta\right) + i \sin\left( \frac{b}{2} \beta \right)\right) \right]\right|,
\label{EQfunctionspacepositive}
\end{eqnarray}
where we can see that  if $\cos \left( \frac{b}{2} \beta \right) > 0$ then $\int_{-d}^{d} | v(\alpha + i \beta) | d\beta \to 0$ as $\alpha \to \infty$.

Take $\gamma_1 = \dfrac{b}{2}$, and we see that if $|\beta| < \dfrac{\pi}{2 \gamma_1}$ then  $v \in \mathbf{B}_p(\mathcal{D}_{d})$.

Conversely, as $x \to -\infty$ and utilizing the facts that:
\begin{equation}
\exp\left[ae^{b x}-ce^{-d x}\right] \to 0 \;\;\textrm{as}\;\; x \to -\infty \quad \textrm{and} \quad
\log(1+u) = \displaystyle \sum_{n=0}^{\infty} \frac{(-1)^{n}u^{n+1}}{n+1} \;\;\textrm{as}\;\; u \to 0,
\end{equation}
we obtain the following expansion:
\begin{eqnarray}
\log \left[\exp\left(a e^{b x}-ce^{-d x}\right)+1\right] & \sim & \log \left[ \exp\left( -ce^{-d x} \right) +1  \right]
\nonumber\\& \sim & \displaystyle \sum_{n=0}^{\infty} \frac{(-1)^{n}\left[\exp\left(-ce^{-d x}\right)\right]^{n+1}}{n+1}
\nonumber\\& \sim & \exp\left[-ce^{-d x}\right] \quad \textrm{as} \quad x \to -\infty.
\label{EQasymptoticexpansion02}
\end{eqnarray}

Composing the transformation with the asymptotic expansion of the wavefunction leads to:
\begin{equation}
v(x) \sim a \left[\exp\left(-ce^{-d x}\right)\right]^{r_1} \quad \textrm{as} \quad x \to -\infty.
\end{equation}
Then, as $ x \to -\infty$:
\begin{eqnarray}
| v(x)| & \sim & \left| A \exp\left[ c r_1 e^{-d \, x}\right]\right|
\nonumber\\& \sim & \left|A \exp \left[ C e^{-d(\alpha + i \beta)}\right]\right|
\nonumber\\& \sim & \left|A \exp \left[ Ce^{-d \alpha}\left(\cos(d \beta) - i\sin(d \beta )\right)\right]\right|
\nonumber\\& \sim & A \exp\left[ Ce^{-d \alpha} \cos(d \beta) \right].
\label{EQfunctionspacenegative}
\end{eqnarray}

Now, take $\gamma_2 = d $. If $ \cos ( \gamma_2 \,\beta) >0 $, then $ \int_{-d}^{d} | v(\alpha + i \beta) | d\beta \to 0$ as $\alpha \to \infty$  and  $v \in \mathbf{B}_2(\mathcal{D}_{\frac{ \pi}{2 \gamma_2}})$.

Now, if we choose the width of $\mathcal{D}_{d}$  so that the conditions on $\gamma_i$ for $i = 1,2$ hold, the transformation satisfies the requirements in Definition \ref{Deffunctionspace}.

By taking $ \gamma = \max \{ \gamma_1, \gamma_2 \}$ and $ d = \dfrac{\pi}{2 \gamma}$ , we ensure that $ v(x) \in \mathbf{B}_2(\mathcal{D}_{d})$ and is well suited to Sinc approximation.

The matrices involved in the calculation become ill-conditioned. This is to be expected as the Schr\"{o}dinger equation produces eigenvalues that grow unboundedly. We notice that the numerical blow ups correspond to the increasing condition number. As the genreal transformation $\tilde{\phi}(x)$~\eqref{EQtildehix} includes four arbitrary positive parameters $a$, $b$, $c$ and $d$, it allows for more freedom in tailoring the (optimal) transformation for the potential. We define our optimal parameters to be those that increase our numerical stability. Our potential has an algebraic singularity at $x = 0$ resulting in a significant numerical instability. To find the optimal parameters, we started with $a = c = 0.5$ and $b = d =1$, then tested combinations of different values of the parameters. We noticed that increasing the values of the parameters to 1 significant digit vastly improved stability. We anticipate that further optimization of the transformation parameters will produce further numerical stability. However, finding the optimal combination is quite costly, as we are optimizing a non-linear function with 4 input values. In our calculations, we used $a = 1.05, b = 1.30, c = 1.20 $ and $d= 0.94$. As can be seen from Figure~\ref{StabilityFigure2}, implementing the generalized transformation $\tilde{\phi}(x)$ improves considerably the numerical stability of the method.

\subsection{Scaling Factor}
To improve the stability as well as the convergence of the method, we introduce a scaling factor in the DESCM.

\begin{corollary}
Scaling the transformed energy eigenvalue problem using $x = \tau y$ with $\tau \not= 0$ will transform the computed eigenvalues by $E = \dfrac{\widetilde{E}}{ \tau^2} $ where $E$ is the original eigenvalue and $\widetilde{E}$ is the energy eigenvalue of the scaled problem.
\label{scaling}
\end{corollary}

\begin{proof}
Consider the potential $V(x) = \displaystyle \sum_{j=-i}^n a_j x^{j}$ and the vector $ \mathbf{a} : = [a_{-i}, a_{-i+1},...,a_n]$  consisting of the coefficients of the potential. Consider also the vector $\mathbf{x} = [x^{-i},x^{-i+1},...,x^n]$. Recognize that the energy eigenvalues are functions of the coefficients of the potential. We can thus write $E(\mathbf{a})$.

We write the problem in the following form:
\begin{equation}
- \psi '' (x) + (\mathbf{a} \cdot \mathbf{x})  \psi(x) = E(\mathbf{a}) \psi (x).
\end{equation}

Implementing the change of variable $x = \tau y$ with $\tau \in \mathbb{R}$ and $\tau \not= 0$, and using:
\begin{equation}
\frac{d^2}{dx^2}\psi(\tau y) = \frac{1}{\tau ^2} \frac{d^2}{dy^2} \psi(y),
\end{equation}
leads to:
\begin{equation}
 \frac{\psi '' ( \tau y  )}{ \tau ^2} + (\mathbf{b} \cdot \mathbf{y})  \psi(\tau y) = E(\mathbf{b}) \psi(\tau  y),
\label{EQlineartransform1}
\end{equation}
where $\mathbf{b} = [ \tau^{-i} a_i, \tau^{-i+1} a_{-i+1}...,\tau^n a_n]$ and $\mathbf{y} = [y^{-i},y^{-i+1},...,y^n]$.

If we let $\mathbf{c} = [ \tau^{-i+2} a_i, \tau^{-i+2} a_{-i+1}...\tau^{n+2} a_n]$, then \eqref{EQlineartransform1} becomes:
\begin{equation}
- \psi '' (\tau y) + \mathbf{c} \cdot \mathbf{y} =E(\mathbf{c}) \psi(\tau y) = \tau^2 E(\mathbf{a}) \psi(x).
\label{EQlineartransform2}
\end{equation}

We can recover the energy eigenvalues corresponding to $a_j$ by noticing that:
\begin{equation}
E(\mathbf{a}) = \frac{E(\mathbf{c})}{ \tau^2},
\label{EQlineartransform3}
\end{equation}
as desired.
\end{proof}

The scaling vastly improves the number of convergent eigenvalues found by the method. We fixed the matrix size at $ 201 \times 201$, and computed the number of convergent eigenvalues with and without the scaling factor for each transformation and we report the substantial increase in the number of convergent eigenvalues found as can be seen from Table~\ref{EigenvalueTable}. In Figure \ref{StabilityFigure2}, we used $1001 \times 1001$ matrix illustrating the increased stability of the method when the scaling factor is used. In Figure~\ref{Comparison}, we plot the convergence rates towards the ground state eigenvalues using both transformations, $\phi(x)$~\eqref{EQphix} and $\tilde{\phi}(x)$~\eqref{EQtildehix}, and the scaling factor. There, we see a mixture of increased and decreased convergence when applying the scaling. However, the increased stability remains consistent.

We notice the decrease in convergence speed to the ground state as $\tau$ is increased. In the following result, we show that the scaling factor $\tau$ cannot be increased indefinitely.
\begin{corollary}
Assume that the function $f$ satisfies the following:
\begin{equation}
|f(x)| \leq A \left \{
\begin{array} {lll}
\exp \left(- B e^{\gamma_L |x|}\right) & for & x \in (-\infty , 0] \\[0.25cm]
\exp \left(- B e^{\gamma_R |x|}\right) & for & x \in [0,\infty),
\end{array}
\right.
\end{equation}
for some positive constants $\gamma_L$, $\gamma_R$ and $B$. Take $\gamma = \max \{ \gamma_L, \gamma_R \}$. Assume further that $f \in \mathbf{B}_2(\mathcal{D}_{d})$ with $ d < \dfrac{\pi}{2 \gamma}$. Let $E_{N}$ be the error between the second derivative of $f$ and its truncated Sinc Expansion. Take:
\begin{equation}
h = \frac{W(\pi d \gamma N / B)}{\gamma N}.
\end{equation}

Finally, assume that after introducing scaling in the Sinc transformation, there exists a constant $\alpha \in \mathbb{R} $ such that $\dfrac{B_S}{ \tau^{\alpha}} = B$. Then, the scaling factor $\tau$ that decreases $E_N$ is bounded in $\mathbb{R}$.
\end{corollary}
\begin{proof}

We will first show existence of the upper bound. Assume that there is no bound, $\tau_{optimal}$. Then increasing the scaling factor past $\tau_{optimal}$ will continually decrease $E_N$.

To derive a contradiction, we begin by writing the error between the truncated Sinc expansion and the second derivative of $f$ as the sum of the truncation and sampling errors:
\begin{equation}
E(N) = \left\| f''(x) - \displaystyle \sum_{k = -\infty}^{\infty} f(kh)S''(k,h)(x) + \displaystyle \sum_{|k| > N }^{\infty}f(kh)S''(k,h)(x)\right\|.
\end{equation}
 Now, utilizing the upper bound derived in the proof of theorem 3.5.1 of \cite{Stenger-23-165-81}, we find that:
\begin{equation}
\left\|  \displaystyle \sum_{|k| > N }^{\infty} f(kh) S''(k,h)(x) \right\| \leq 2 \left( \frac{\pi}{5} \right)^{\frac{1}{2}} \left( \frac{\pi}{h} \right)^{\frac{3}{2}}   \displaystyle \sum_{|k| > N }^{\infty} \left|f(kh)\right|.
\end{equation}

Next, utilizing the assumption on the bound of $f$ and the fact that a sum of a non-negative function is dominated by an integral, we find:
\begin{equation}
\displaystyle \sum_{|k| > N }^{\infty} \left|f(kh)\right| \leq
\left( \int_{N}^{\infty} \left|A \exp\left( -B \, e^{\gamma_R |x|}\right)\right| dx +
\int_{-\infty}^{-N} \left|A \exp\left( -B \, e^{\gamma_L |x|}\right) \right|  dx \right).
\end{equation}

For brevity, write:
\begin{equation}
I_1 = \int_{N}^{\infty} \left|A \exp \left( -B\, e^{\gamma_R |x|}\right)\right| dx \quad \textrm{and} \quad
I_2 = \int_{-\infty}^{-N} \left|A \exp\left( -B \,e^{\gamma_L |x|}\right) \right| dx.
\end{equation}

Then:
\begin{equation}
\left\|  \displaystyle \sum_{|k| > N}^{\infty} f(kh)S''(k,h)(x) \right\| \leq 2 \left( \frac{\pi}{5} \right)^{\frac{1}{2}} \left( \frac{\pi}{h} \right)^{\frac{3}{2}} \left( I_1+ I_2 \right).
\end{equation}

Now, taking account of the absolute value of $x$ allows us to write:
\begin{equation}
 I_1 = \int_{N}^{\infty} \left|A \exp\left( -B \, e^{\gamma_R x} \right) \right| dx \quad \textrm{and}
 \quad I_2 = \int_{-\infty}^{-N} \left|A \exp\left(-B \, e^{- \gamma_L x}\right)\right|dx.
\end{equation}

These integrals can be transformed into exponential integrals using the following changes of variables:
\begin{equation}
 u_1 = B\, e^{\gamma_R x} \;\;\Rightarrow \;\; du_1 = \gamma_R u_1 dx \quad \textrm{and} \quad  u_2 = B \, e^{- \gamma_L x} \;\;\Rightarrow \;\; du_2 = - \gamma_L u_2 dx,
\end{equation}
leading to:
\begin{equation}
\left\| \displaystyle \sum_{|k| > N}^{\infty} f(kh) S''(k,h)(x) \right\| \leq
\frac{2A \pi^2 }{ h^{\frac{3}{2}} \sqrt{5} } \left[ \frac{1}{\gamma_R}
\int_{B\, e^{\gamma_R \,N}}^{\infty} \left| u_1^{-1} e^{u_1} \right| du_1 + \frac{1}{\gamma_L}
\int_{B \, e^{\gamma_L \, N}}^{\infty} \left| u_2^{-1} e^{u_2} \right| du_2 \right].
\end{equation}

Then, inputting the first term of the asymptotic expansion for the exponential integral, we find:
\begin{equation}
\left\|\displaystyle \sum_{|k| > N }^{\infty} f(kh) S''(k,h)(x)\right\| \leq
\frac{2A \pi^2 }{ h^{\frac{3}{2}} \sqrt{5} } \left[ \frac{ e^{-u_1\,N}}{\gamma_R u_1\,N} +
\frac{e^{u_2\,N}}{\gamma_L u_2\, N} \right].
\end{equation}

Now, we combine all constant terms and realize that $h$ is dependent on the parameter $B_s$ via the Lambert W function.  If $|x|< \dfrac{1}{e}$, then the Lambert W function of $x$  converges to the Taylor Series:
\begin{equation}
W(x) = \displaystyle \sum_{n=1}^{\infty} \frac{(-n)^{n-1}}{n!} \, x^n.
\end{equation}

Now, we notice that:
 \begin{equation}
\tau > \left( \frac{ \pi d \gamma N e}{B} \right)^{\frac{1}{\alpha}} \quad \Longrightarrow \quad \frac{\pi d \gamma N}{B_s}  < \frac{1}{e}.
\end{equation}

As we are assuming that we can increase $\tau$ indefinitely, this condition will eventually be satisfied.
Now, we use the Taylor series expansion for the W Lambert function and find:
\begin{equation}
h \sim \frac{(\pi d \gamma N / B_s)}{\gamma N}.
\end{equation}

Collecting the constants in the error bound yields:
\begin{equation}
\left\|\displaystyle \sum_{|k| > N }^{\infty}f(kh)S''(k,h)(x)\right\| \leq \frac{2A_{N,\gamma_R,\gamma_L}  \pi ^2}{ \left(  \frac{\pi d \gamma N}{ \tau^{\alpha} B } \right) ^{1.5} },
\end{equation}
where $A_{N,\gamma_R,\gamma_L}$ is a constant that depends on $N,\gamma_r $ and $ \gamma_L$ via the expansion of the exponential integrals.
Simplifying gives:
\begin{equation}
\left\|\displaystyle \sum_{|k| > N }^{\infty}f(kh)S''(k,h)(x) \right\| \leq \frac{2A_{N,\gamma_R,\gamma_L} \pi ^2 (\tau^{\alpha} B)^{1.5} }{ \left(  \pi d \gamma N \right) ^{1.5} },
\end{equation}
and we see that the truncation error increases algebraically with $\tau$.

Now, we investigate the sampling error. In \cite{Lundin1979}, the authors present the following bound:
\begin{equation}
\left\| f''(x) - \displaystyle \sum_{k = -\infty}^{\infty}f(kh)S''(k,h)(x) \right\| \leq \frac{2eN(f,2,D_{d})}{2 \pi \sinh(\frac{\pi d}{h})} \left( \frac{ \pi}{h} \right)^2.
\end{equation}

By inserting the same approximation for the step size, we arrive at:
\begin{equation}
\left\|f''(x) - \displaystyle \sum_{k = -\infty}^{\infty}f(kh)S''(k,h)(x) \right\| \leq \frac{eN(f,2,D_{d})}{ \pi } \left( \frac{ \tau^{\alpha}B}{d \gamma N} \right)^2 \frac{1}{\sinh \left( \frac{\tau^{\alpha}B}{\gamma N } \right)}.
\end{equation}

Simplifying and collecting constants, leads to:
\begin{equation}
\left\|f''(x) - \displaystyle \sum_{k = -\infty}^{\infty}f(kh)S''(k,h)(x)\right\| \leq \frac{eN(f,2,D_{d})B^2}{\pi (d \gamma N)^2} \tau^{2 \alpha} \frac{1}{ \sinh \left( \frac{\tau^{\alpha} B}{\gamma N} \right)}.
\end{equation}
So here the sampling error decays to 0 exponentially as $\tau$ increases.

As the sampling error is bounded below by 0, there will be a $\tau_{0}$ where it becomes negligible. For  $\tau > \tau_0$, the error in the approximation of the function will become dominated by the truncation error. However, as the truncation error grows as an algebraic power of $\tau$, indefinitely increasing $\tau$ will increase the total error bound. This is a contradiction to the assumption that we can indefinitely increase $\tau$ to decrease the error and gives the boundedness property of the interval.

\end{proof}

It is important to note that finding an expression for the upper bound of the interval requires an estimation of the contour integral $N(f,2,D_d)$ which has proven to be non-trivial analytically and numerically.

\section{Numerical discussion and results}
We use the DESCM to find energy eigenvalues of the anharmonic Coulomb potential. The coding language Matlab is used for the calculation. We define the relative error between known eigenvalues $E$  and numerical eigenvalues ${\mathcal E}$  as:
\begin{equation}
\textrm{Relative Error} = \frac{|E - {\mathcal E}|}{| E |}.
\end{equation}
When moving to higher order energy eigenvalues or potentials without known analytic solutions, we use the following formula to show convergence of the numerical eigenvalues:
\begin{equation}
\textrm{Relative Error Approximation} = \frac{|{\mathcal E}_i(n+1) - {\mathcal E}_i(n)|}{| {\mathcal E}_i(n+1) |},
\end{equation}
where ${\mathcal E}_i(n+1)$ denotes the $(n+1)$th approximation of the $ith$ energy eigenvalue.

We use the Lambert W function to find the mesh size $h$, as well as the transformations $\phi(x)$~\eqref{EQphix} and $\tilde{\phi}(x)$~\eqref{EQtildehix}, to transform the eigenvalue problem. To illustrate the convergence of our method, we compute the eigenvalues of potentials that have known analytic solutions~\cite{Lund1984,Chaudhuri1995a}. The potentials which we use are the followings:
\begin{equation}
\begin{array}{ccccccc} \sphline
V_1 & = & \frac{2}{x^2}-\frac{16}{x}+2x+\frac{x^2}{16} & \Longrightarrow & \mathbf{E}_0 & = & -\frac{59}{4} \\[0.25cm]
V_2 & = & \frac{6}{x^2}-\frac{24}{x}+2x+\frac{x^2}{16} & \Longrightarrow & \mathbf{E}_0 & = & -\frac{57}{4} \\[0.25cm]
V_3 & = & \frac{15}{4x^2}-\frac{20}{x}+2x+\frac{x^2}{16} & \Longrightarrow & \mathbf{E}_0 & = & -\frac{58}{4} \\[0.25cm]
V_4 & = & \frac{35}{4x^2}-\frac{28}{x}+2x+\frac{x^2}{16} & \Longrightarrow & \mathbf{E}_0 & = & -14 \\[0.25cm]
V_5 & = & \frac{2}{x^2} + x^2 & \Longrightarrow & \mathbf{E}_0 & = & 5 \\[0.25cm]
V_6 & = & \frac{3}{4x^2} + x^2 & \Longrightarrow & \mathbf{E}_0 & = & 4. \\\sphline
\end{array}
\end{equation}

The refined DESCM provides increased convergence speed when compared with the SESCM presented in \cite{Lund1984}. To compare the two methods, we implement the SESCM following the procedure in \cite{Lund1984}. The improved convergence speed offered by the refined DESCM is predicted in theoretical work done by Sugihara and others \cite{Mori2001a,Sugihara2002a,Tanaka2009}. As we have shown that the solution of the Schr\"{o}dinger equation is well suited to the DESCM, our results are remarkable. The convergence of both methods is plotted in Figure \ref{SingleDoubleComparison}. For the potentials $V_5$ and $V_6$, we utilize the same single exponential transformation and step size as proposed in \cite{Lund1984}. Moreover, we show that the refinements presented in this work also improve the convergence of the SESCM.

In Table \ref{EigenvalueTable}, we calculate the number of convergent eigenvalues in $100$ iterations for the potentials $V_1,V_2,V_3$ and $ V_4$.  For higher order eigenvalues where the analytic solution is not known, we use the relative error threshold of $5 \cdot 10^{-12}$. We consider a higher order eigenvalue to be found if the relative error approximation is within the acceptable error. Our choice of acceptable error is influenced by the accuracy of the eigensolvers in Matlab as well as the presence of round off error. In this table, the improvement resulting from utilizing the generalized transformation $\tilde{\phi}(x)$~\eqref{EQtildehix} and the introduction of the scaling factor is obvious.

In Table \ref{ConvergenceTable}, we present the evolution of the convergence for increasing matrix size for the potential $V_1$. We see convergence towards the known ground state eigenvalue as well as the convergence towards the first and second excited states. However, we would like to be able to compute arbitrarily many energy eigenvalues. This will require dealing with matrices of increasingly large size. Further, these matrices become more and more ill-conditioned as they grow. In fact, for the potential $V_1$, when using the transformation $\phi(x)$~\eqref{EQphix}, numerical blow ups occur for a $141 \times 141 $ matrix and higher. We plot the condition number of the generalized eigenvalue problem, and notice that the numerical blow ups occur as the condition number of the eigenvalue problem passes $10^{16}$. This increase in the condition number is to be expected as the energy eigenvalues of the system grow without bound. The scaling factor presents a simple way to improve stability of the method and is evidenced in Figure \ref{StabilityFigure2}. The Figure shows the improved convergence of the transformation $\tilde{\phi}(x)$ when compared with $\phi(x)$. We also see the vastly improved stability of the scaled transformation.

For the potentials in Figure \ref{Comparison}, we use a $201 \times 201 $ matrix and we show a further comparison of $\phi(x)$ and $\tilde{\phi}(x)$ for the potentials $V_1$, $V_2$, $V_3$ and $V_4$. We note the improved convergence speed and number of convergent eigenvalues.

\section{Conclusion}
In this paper, we apply the DESCM method to the Schr\"{o}dinger equation with an anharmonic Coulombic potential. This potential presents several numerical difficulties, including a singularity at $x = 0$. The DESCM proves to be a powerful choice for computing the energy eigenvalues and produces convergence towards known eigenvalues quickly. Further, we show that for the Coulombic potential, the double exponential transformation is the correct choice of transformation. Further, we introduced an improvement of the numerical stability as well as the convergence of the DESCM. The scaling factor that utilizes the symmetry of the eigenvalues is, to our knowledge, a novel suggestion that vastly improves stability and increases convergence. Our numerical results imply that the instability is due to the problem becoming ill-conditioned for large matrix sizes. Future work will include implementing preconditioning methods in the generalized eigenvalue problem as well as exploring other methods of increasing stability.

\section{Numerical tables and figures}
\clearpage

\begin{figure} [h!]
\begin{tabular} {l}
\hspace*{-0.85cm} \includegraphics[trim=0 150 0 150,clip,width=0.525\textwidth]{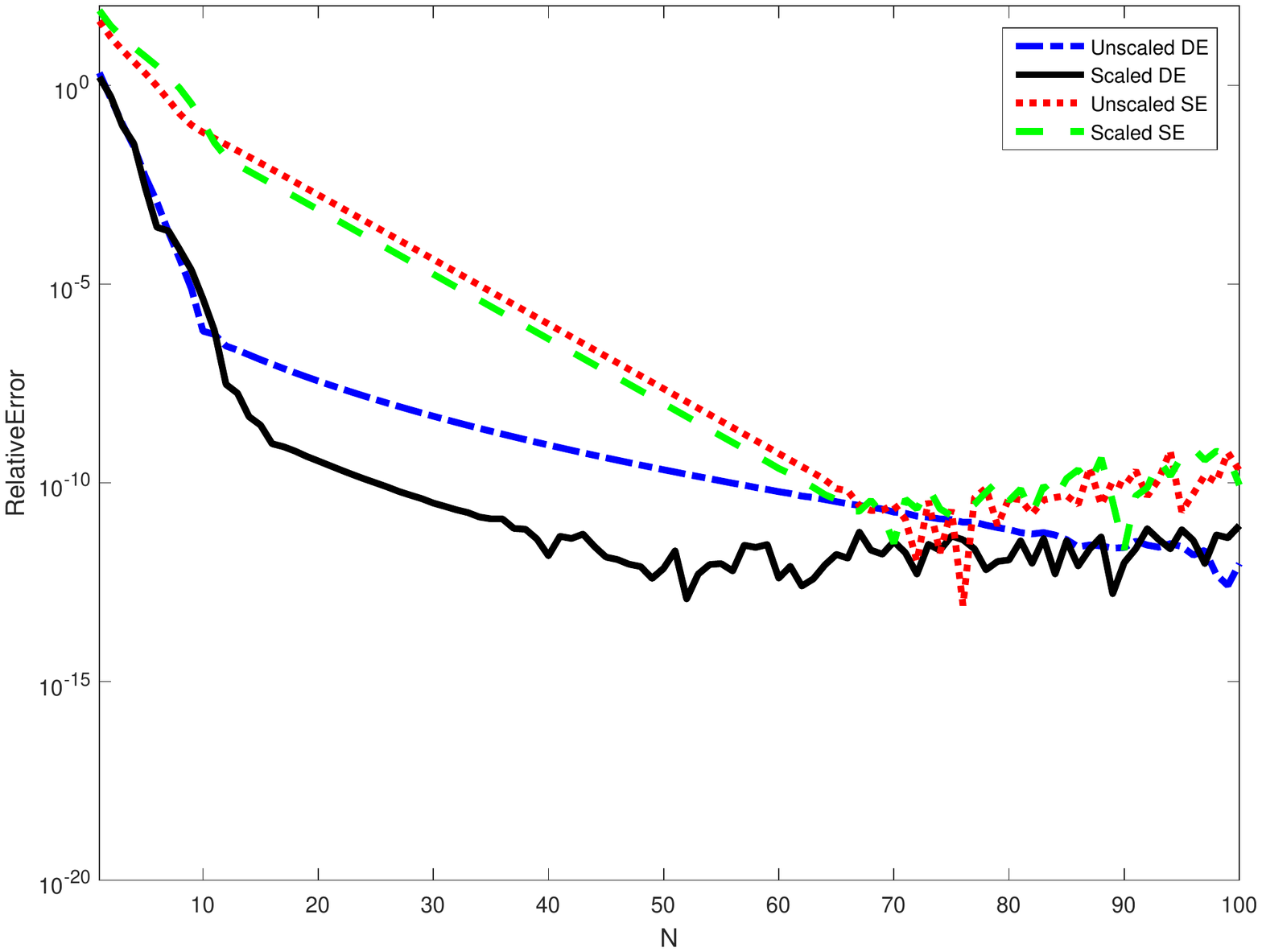}~\includegraphics[trim=0 150 0 150,clip,width=0.525\textwidth]{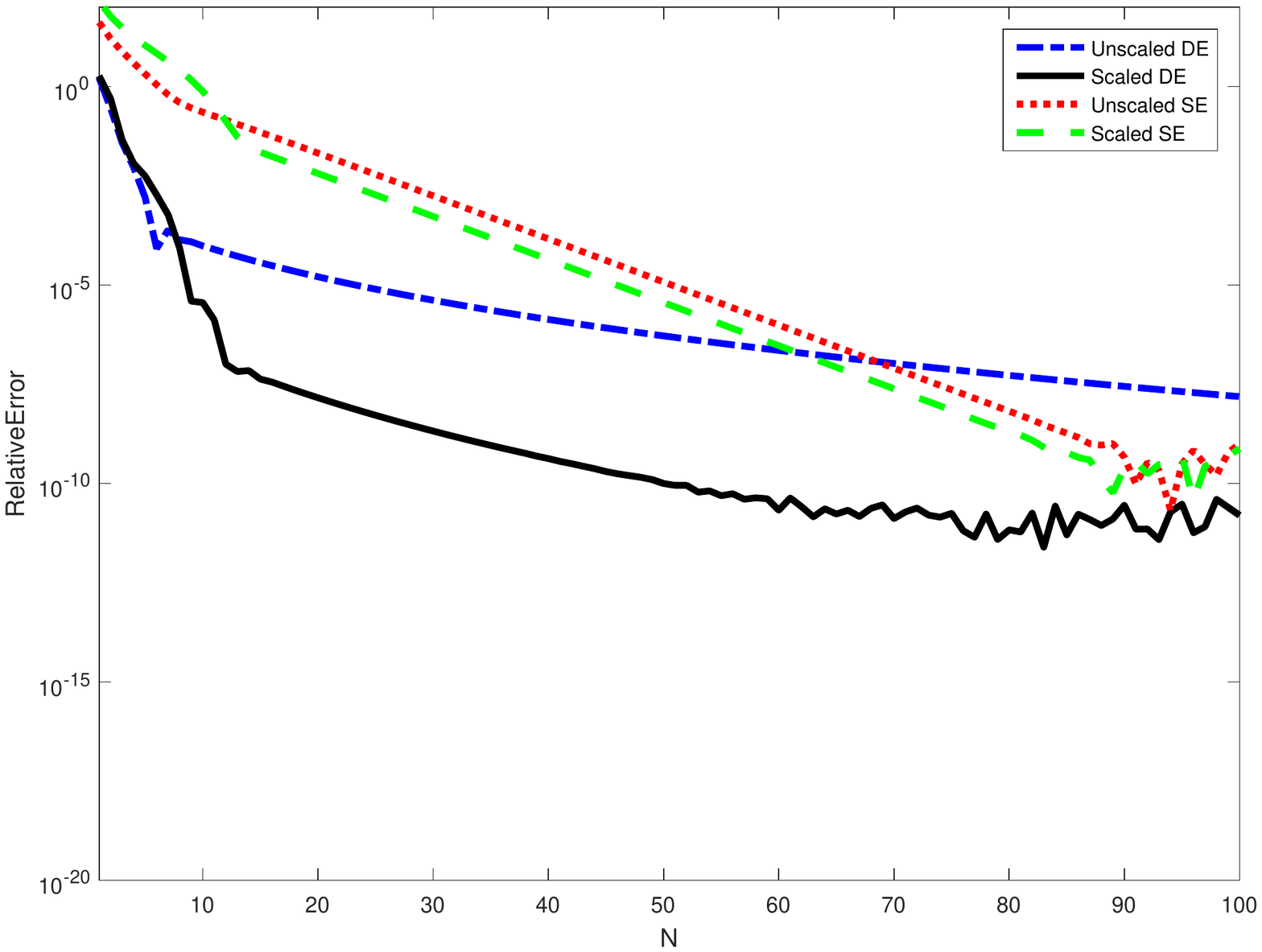} \\
[-0.5cm]
\hspace*{-1.25cm} ~~~~~~~~~~~~~~~~~~~~~~~~~~~~~~~~(a)~~~~~~~~~~~~~~~~~~~~~~~~~~~~~~~~~~~~~~~~~~~~~~~~~~~~~~~~~~~~~~(b)
\end{tabular}
\caption{
Comparison of the DESCM and the SESCM. (a) represents the convergence of the DESCM and the SESCM towards known eigenvalues of the potential $V_5$. (b) represents the convergence of the DESCM and the SESCM towards the eigenvalues of the potential $V_6$. The scaled plots correspond to the convergence diagrams using the scaling factor $\tau$. (a) uses a scaling factor of $\tau = 0.75$. (b) uses a scaling factor of $\tau = 0.55$.   }
\label{SingleDoubleComparison}
\end{figure}

\vspace{0.5cm}

\begin{table} [h!]
\caption{ The amount of convergent eigenvalues computed in 100 iterations for both transformations $\phi(x)$~\eqref{EQphix} and $\tilde{\phi}(x)$~\eqref{EQtildehix}. $\tau$  denotes the scaling factor used in the calculation.}
\begin{tabular*}{\hsize}{@{\extracolsep{\fill}}ccccc} \hline
Potential &  $\phi(x)$, $\tau = 1$ & $\phi(x)$, $\tau = 1.75$ & $\tilde{\phi}(x)$, $\tau = 1$   &  $\tilde{\phi}(x)$, $\tau = 1.75$	     \\\hline
$ V_1(x) $ & 8 & 22 & 22 & 36 \\
$ V_2(x)$ & 9 & 23 & 21 & 35\\
$ V_3(x) $ & 8 & 20 & 19 & 37 \\
$V_4(x) $ & 9 & 22 &  20 & 34 \\ \hline
\end{tabular*}
\label{EigenvalueTable}
\end{table}

\vspace{-0.25cm}
\begin{table}[h!]
\caption{ Numerical calculations for the ground states and first two excited states of the potential $ V_1(x)$ Here we used the potential $\tilde{\phi}(x)$~\eqref{EQtildehix} with the scaling factor $\tau = 1.00$. }
\begin{tabular*}{\hsize}{@{\extracolsep{\fill}}cccc}\sphline
  $N$ &      $\tilde{E}_{0}(N)$    &  $\tilde{E}_{1}(N)$        &      $\tilde{E}_{2}(N)$   \\\hline
10 & -14.7499998222764 & -4.09661939808125 & 1.13533983977096\\
15 & -14.7499999935935 & -4.09661597228020 & 1.13571953379622 \\
20 & -14.7499999989570 & -4.09661597504977 & 1.13571957570939 \\
25 & -14.7499999997938 & -4.09661597544138 & 1.13571957544272 \\
30 & -14.7499999999506 & -4.09661597551624 & 1.13571957539198 \\
35 & -14.7499999999867 & -4.09661597553405 & 1.13571957537878 \\
40 & -14.7499999999960 & -4.09661597553543 & 1.13571957537729 \\
45 & -14.7500000000008 & -4.09661597553923 & 1.13571957537739 \\
50 & -14.7499999999961 & -4.09661597554020 & 1.13571957537189 \\\hline
\end{tabular*}
\label{ConvergenceTable}
\end{table}

\clearpage

\begin{figure} [h!]
\begin{tabular} {l}
\hspace*{-1.25cm} \includegraphics[trim=0 150 0 150,clip,width=0.55\textwidth]{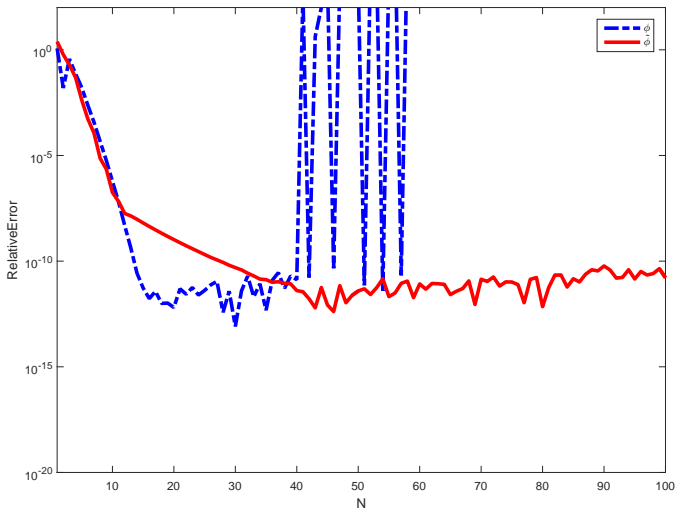} \includegraphics[trim=0 150 0 150,clip,width=0.55\textwidth]{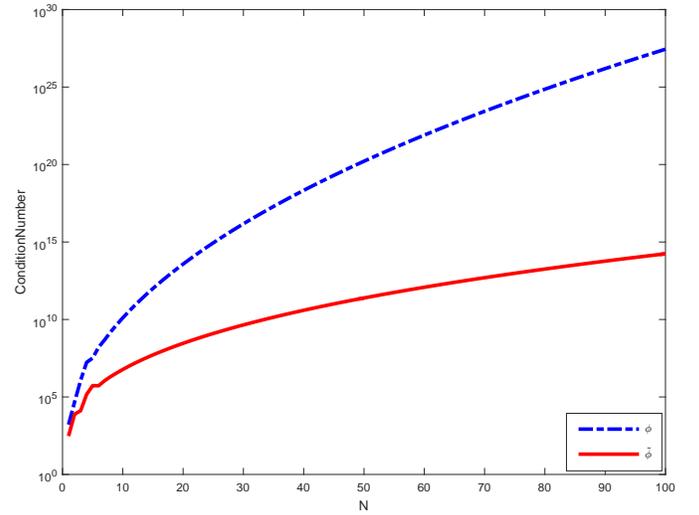}\\[-0.5cm]
\hspace*{-1.25cm} ~~~~~~~~~~~~~~~~~~~~~~~~~~~~~~~~(a)~~~~~~~~~~~~~~~~~~~~~~~~~~~~~~~~~~~~~~~~~~~~~~~~~~~~~~~~~~~~~~(b) \\
\hspace*{-1.50cm} \includegraphics[trim=0 150 0 150,clip,width=0.58\textwidth]{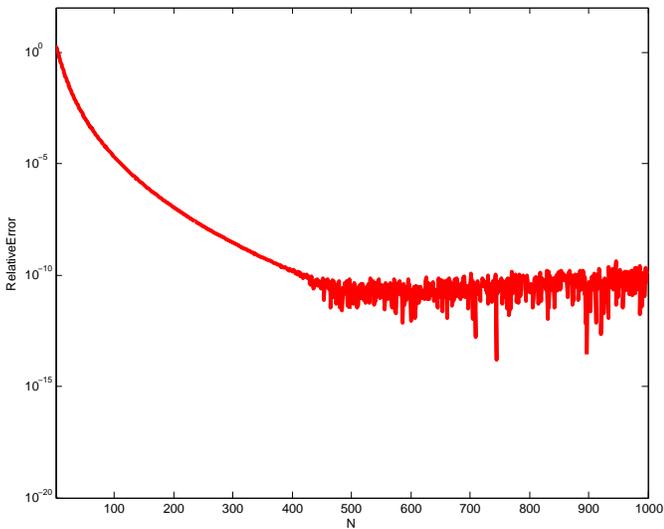} \hspace*{-0.35cm}\includegraphics[trim=0 150 0 150,clip,width=0.58\textwidth]{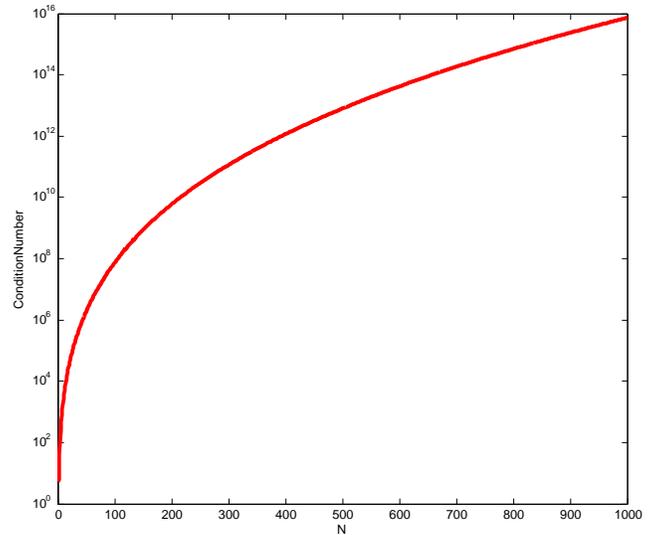} \\[-0.5cm]
\hspace*{-1.25cm} ~~~~~~~~~~~~~~~~~~~~~~~~~~~~~~~~(c)~~~~~~~~~~~~~~~~~~~~~~~~~~~~~~~~~~~~~~~~~~~~~~~~~~~~~~~~~~~~~~(d)
\end{tabular}
\caption{ The improved numerical stability of the DESCM for the potential $V_1$. (a) shows the convergence of the method with $\phi(x)$~\eqref{EQphix} and $\tilde{\phi}(x)$~\eqref{EQtildehix} over 100 iterations. (b) compares the condition numbers of the different transformation for the generalized eigenvalue problem. (c) shows the convergence and stability of the generalized transformation following the introduction of a scaling factor $\tau = 3.00$ over 1000 iterations. (d) shows the condition number of the scaled generalized transformation. }
\label{StabilityFigure2}
\end{figure}

\begin{figure} [h!]
\begin{tabular} {l}
\hspace*{-1.25cm} \includegraphics[trim=0 150 0 150,clip,width=0.55\textwidth]{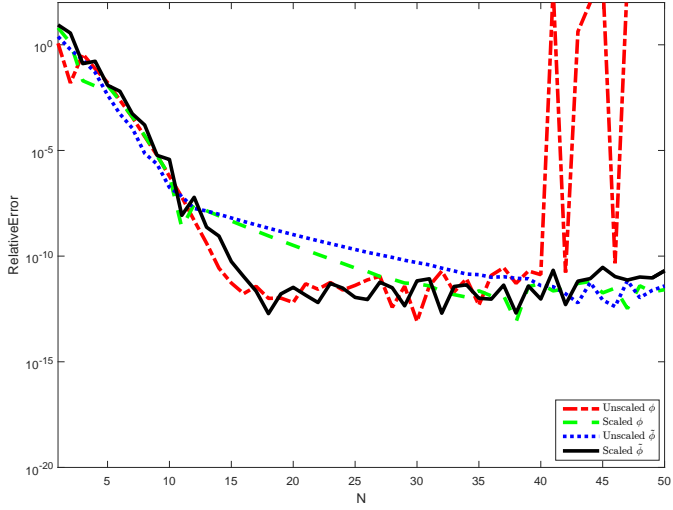}~\includegraphics[trim=0 150 0 150,clip,width=0.55\textwidth]{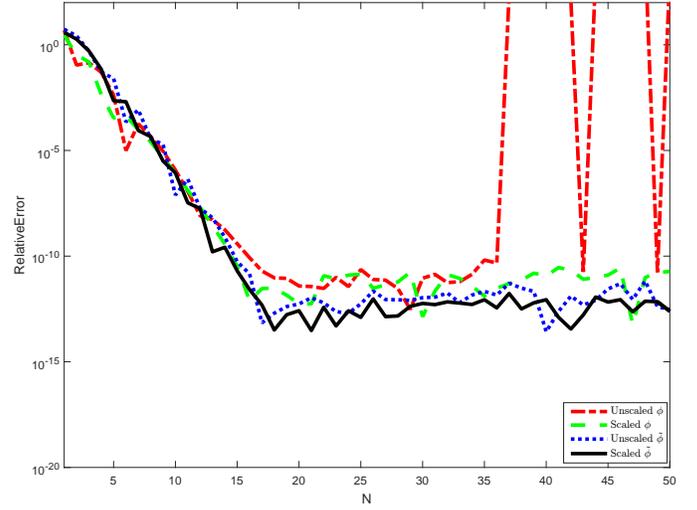} \\
[-0.5cm]
\hspace*{-1.25cm} ~~~~~~~~~~~~~~~~~~~~~~~~~~~~~~~~(a)~~~~~~~~~~~~~~~~~~~~~~~~~~~~~~~~~~~~~~~~~~~~~~~~~~~~~~~~~~~~~~(b) \\
\hspace*{-1.25cm} \includegraphics[trim=0 150 0 150,clip,width=0.55\textwidth]{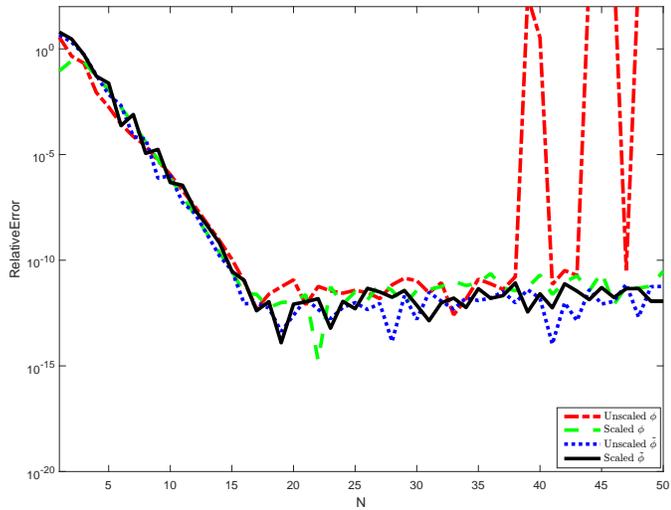}~\includegraphics[trim=0 150 0 150,clip,width=0.55\textwidth]{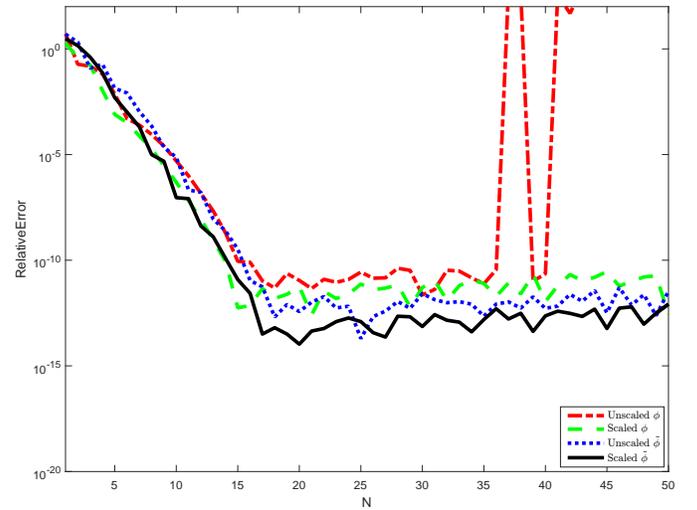}\\
[-0.5cm]
\hspace*{-1.25cm} ~~~~~~~~~~~~~~~~~~~~~~~~~~~~~~~~(c)~~~~~~~~~~~~~~~~~~~~~~~~~~~~~~~~~~~~~~~~~~~~~~~~~~~~~~~~~~~~~~(d)
\end{tabular}
\caption{ The convergence of the DESCM transform towards known ground state energy eigenvalues.  In all figures, we have $\phi(x)$~\eqref{EQphix} and $\tilde{\phi}(x)$~\eqref{EQtildehix}. The scaled plots correspond to the convergence diagrams using the scaling factor $\tau$. We chose the scaling factors, $\tau_1,\tau_2$, to increase convergence and avoid blow up. (a) is the potential $V_1(x)$ with $\tau_1 = 1.75,\tau_2 = 0.6$ for the scaled transformations. (b) is the potential $V_2(x)$ with $\tau_{1,2} = 1.20$ for the scaled transformation. (c) is the potential $V_3(x)$ with $\tau_1 = 1.30,\tau_2 = 0.85 $ for the scaled transformation. (d) is the potential $V_4(x)$ with $\tau_{1,2} = 1.45$ for the scaled transformation. }

\label{Comparison}
\end{figure}

\clearpage

\begin{thebibliography}{10}

\bibitem{Roychoudhury1990a}
R.~Roychoudhury, Y.~P. Varshni, and M.~Sengupta.
\newblock {Family of exact solutions for the Coulomb potential perturbed by a
  polynomial in r}.
\newblock {\em Phys. Rev. A}, 42(1):184--192, 1990.

\bibitem{Ozer2003}
O.~Ozer and B.~Gonul.
\newblock {New exact treatment of the perturbed Coulomb interactions}.
\newblock {\em Mod. Phys. Lett. A}, 18(36):2581--2586, 2003.

\bibitem{Chaudhuri1995a}
R.~N. Chaudhuri and M.~Mondal.
\newblock {Eigenvalues of anharmonic oscillators and the perturbed Coulomb
  problem in N-dimensional space}.
\newblock {\em Phys. Rev. A}, 52(3):1850--1856, 1995.

\bibitem{Ikhdair2007}
S.~M. Ikhdair and R.~Sever.
\newblock {An alternative simple solution of the sextic anharmonic oscillator
  and perturbed Coulomb problems}.
\newblock {\em Int. J. Mod. Phys. C}, 18(10):1571--1581, 2007.

\bibitem{Fernandez2008a}
F.~M. Fern\'{a}ndez.
\newblock {Calculation of bound states and resonances in perturbed Coulomb
  models}.
\newblock {\em Phys. Lett. A}, 372(17):2956--2958, 2008.

\bibitem{Lund1984}
J.~R. Lund and B.~V. Riley.
\newblock {A sinc-collocation method for the computation of the eigenvalues of
  the radial Schr\"{o}dinger equation}.
\newblock {\em IMA J. Numer. Anal.}, 4(1):83--98, 1984.

\bibitem{Tater1999}
M.~Tater.
\newblock {The Hill determinant method in application to the sextic oscillator:
  limitations and improvement}.
\newblock {\em J. Phys. A. Math. Gen.}, 20(9):2483--2495, 1987.

\bibitem{Fernandez1991b}
F.~M. Fern\'{a}ndez.
\newblock {Exact and approximate solutions to the Schr\"{o}dinger equation for
  the harmonic oscillator with a singular perturbation}.
\newblock {\em Phys. Lett. A}, 160(6):511--514, 1991.

\bibitem{Fernandez1989}
F.~M. Fernandez, Q.~Ma, and R.~H. Tipping.
\newblock {Eigenvalues of the Schrodinger equation via the Riccati-Pade
  method}.
\newblock {\em Phys. Rev. A}, 40(11):6149--6153, 1989.

\bibitem{Stenger-33-85-79}
F.~Stenger.
\newblock A {Sinc-Galerkin} method of solution of boundary value problems.
\newblock {\em Mathematics of computation}, 33:85--109, 1979.

\bibitem{Stenger-23-165-81}
F.~Stenger.
\newblock Numerical methods based on {Whittaker} cardinal, or {Sinc} functions.
\newblock {\em SIAM Rev.}, 23:165--224, 1981.

\bibitem{Stenger-121-379-00}
F.~Stenger.
\newblock Summary of {Sinc} numerical methods.
\newblock {\em Journal of Computational and Applied Mathematics}, 121:379--420,
  2000.

\bibitem{Gaudreau2014}
P.~Gaudreau, R.~M. Slevinsky, and H.~Safouhi.
\newblock {The double exponential sinc-collocation method for computing
  eigenvalues of Sturm-Liouville problems and applications to anharmonic
  oscillators}.
\newblock {\em Ann. Phys.}, 360:520--538, 2015.

\bibitem{Gaudreau2014a}
P.~Gaudreau, R.M. Slevinsky, and H.~Safouhi.
\newblock {The Double Exponential Sinc Collocation Method for Singular
  Sturm-Liouville Problems}.
\newblock {\em arXiv:1409.7471v2}, 2014.

\bibitem{Al-Khaled-283-2001}
K. Al-Khaled.
\newblock {Sinc numerical solution for solitons and solitary
Waves}.
\newblock {\em J. Comput. Appl. Math.}, 130:283--292, 2001.

\bibitem{Al-Khaled-245-2001}
K. Al-Khaled.
\newblock {Numerical study of Fisher's reaction-diffusion equa-
tion by the Sinc collocation method}.
\newblock {\em J. Comput. Appl. Math.}, 137:245--255, 2001.

\bibitem{Carlson-Dockery-Lund-66-215-97}
T.S. Carlson, J.~Dockery, and J.~Lund.
\newblock A {Sinc}-collocation method for initial value problems.
\newblock {\em Mathematics of computation}, 66:215--235, 1997.

\bibitem{Amore-39-L349-06}
P.~Amore.
\newblock A variational {Sinc} collocation method for strong-coupling problems.
\newblock {\em J. Phys. A: Math. Gen.}, 39:L349--L355, 2006.

\bibitem{ElGamel-Zayed-48-1285-04}
M.~El-Gamel and A.I. Zayed.
\newblock Sinc-{Galerkin} method for solving nonlinear boundary-value problems.
\newblock {\em Computers and Mathematics with application}, 48:1285--1298,
  2004.

\bibitem{ElGamel-Cannon-Zayed-73-1325-03}
M.~El-Gamel, J.R. Cannon, and A.I. Zayed.
\newblock Sinc-{Galerkin} method for solving linear sixth-order boundary-value
  problems.
\newblock {\em Mathematics of computation}, 73:1325--1343, 2003.

\bibitem{Smith-Bogar-Bowers-Lund-28-760-91}
R.C. Smith, G.A. Bogar, K.L. Bowers, and J.~Lund.
\newblock The {Sinc-Galerkin} method for fourth-order differential equations.
\newblock {\em SIAM Journal on Numerical Analysis}, 28:760--788, 1991.

\bibitem{Eggert1987}
N.~Eggert, M.~Jarratt, and J.~Lund.
\newblock {Sinc function computation of the eigenvalues of Sturm-Liouville
  problems}.
\newblock {\em J. Comput. Phys.}, 69(1):209--229, 1987.

\bibitem{Takahasi-Mori-9-721-74}
H.~Takahasi and M.~Mori.
\newblock Double exponential formulas for numerical integration.
\newblock {\em RIMS}, 9:721--741, 1974.

\bibitem{Mori2001a}
M.~Mori and M.~Sugihara.
\newblock {The double-exponential transformation in numerical analysis}.
\newblock {\em J. Comput. Appl. Math.}, 127(1-2):287--296, jan 2001.

\bibitem{Sugihara2002a}
M.~Sugihara.
\newblock {Near optimality of the sinc approximation}.
\newblock {\em Math. Comput.}, 72(242):767--786, 2002.

\bibitem{Tanaka2009}
K.~Tanaka, M.~Sugihara, and K.~Murota.
\newblock {Function classes for successful DE-Sinc approximations}.
\newblock {\em Math. Comput.}, 78(267):1553--1571, 2009.

\bibitem{Sugihara2002b}
M.~Sugihara.
\newblock {Double exponential transformation in the Sinc-collocation method for
  two-point boundary value problems}.
\newblock {\em J. Comput. Appl. Math.}, 149(1):239--250, 2002.

\bibitem{Mori2005}
M.~Mori.
\newblock {Discovery of the double exponential and its developments}.
\newblock {\em Publ. Res. Inst. Math. Sci.}, 41(4):897--935, 2005.

\bibitem{Okayama2013}
T.~Okayama, K.~Tanaka, T.~Matsuo, and M.~Sugihara.
\newblock {DE-Sinc methods have almost the same convergence property as SE-Sinc
  methods even for a family of functions fitting the SE-Sinc methods}.
\newblock {\em Numer. Math.}, 125(3):511--543, 2013.

\bibitem{Tanaka2013}
K.~Tanaka, T.~Okayama, T.~Matsuo, and M.~Sugihara.
\newblock {DE-Sinc methods have almost the same convergence property as SE-Sinc
  methods even for a family of functions fitting the SE-Sinc methods}.
\newblock {\em Numer. Math.}, 125(3):545--568, 2013.

\bibitem{Mori1988}
M.~Mori.
\newblock {An error analysis of quadrature formulas obtained by variable
  transformation}.
\newblock In M. Kashiwara, M. Jimbo, T. Kawai, H. Komatsu, T. Miwa, and M. Morimoto (Editors), {\em Algebr. Anaylsis
  vol.1}, 423--437, Academic Press, London, 1988.

\bibitem{Lundin1979}
L.~Lundin and F.~Stenger.
\newblock {Cardinal-type approximations of a function and its derivatives}.
\newblock {\em SIAM J. Math. Anal.}, 10(1):139--160, 1979.

\end{thebibliography}

\end{document}